\documentclass[preprint,10pt]{elsarticle}

\usepackage{amsmath,amsopn,amssymb,epsfig,amsthm}
\usepackage{multirow}
\usepackage{color}

\definecolor{amethyst}{rgb}{0.6, 0.4, 0.8}
\definecolor{orange}{rgb}{1,0.5,0}
\newtheorem{Algorithm}{Algorithm}[section]
\newtheorem{Theorem}{Theorem}[section]
\newtheorem{Lemma}{Lemma}[section]
\newtheorem{Proposition}{Proposition}[section]
\newtheorem{Remark}{Remark}[section]

\newtheorem{example}{Example}[section]
\newtheorem{Definition}{Definition}[section]
\newtheorem{Corollary}{Corollary}[section]
\newproof{pot}{Proof of Theorem \ref{thm2}}
\newcommand{\bb}{\begin{bmatrix}}
\newcommand{\eb}{\end{bmatrix}}
\newcommand{\bl}[1]{\begin{list}{#1}{\usecounter{bean}}} \newcommand{\el}{\end{list}}
\newcommand{\bel}[1]{\begin{equation} \label{#1}} \newcommand{\eel}{\end{equation}}

\def\r2n2n{\mathbb{R}^{2n\times 2n}}
\def\c2n2n{\mathbb{C}^{2n\times 2n}}

\usepackage{todonotes}

\begin{document}

\date{}
\begin{frontmatter}
\title{
The convergence analysis of an accelerated iteration for solving algebraic Riccati equations
%
}
%

\author{Chun-Yueh Chiang\corref{cor1}\fnref{fn1}}
\ead{chiang@nfu.edu.tw}
\address{Center for General Education, National Formosa
University, Huwei 632, Taiwan.}

\cortext[cor2]{Corresponding author}
\fntext[fn2]{ The second author was supported by the Ministry of Science and Technology of Taiwan under grant 108-2115-M-150-002.}

\date{ }

\begin{abstract}
The discrete-time algebraic Riccati equation (DARE) have extensive applications in optimal control problems. We provide new theoretical supports to the stability properties of solutions to the DARE and reduce the convergence conditions under which the accelerated fixed-point iteration (AFPI) can be applied to compute the numerical solutions of DARE. In particular, we verify that the convergence of AFPI is R-superlinear when the spectral radius of the closed-loop matrix is greater than 1, which is shown by mild assumption and only using primary matrix theories.
Numerical examples are shown to illustrate the consistency and effectiveness of our theoretical results.
\end{abstract}


\end{frontmatter}

\section{Introduction}

The discrete algebraic Riccati equation (DARE) arising in the field of applied mathematics and many classical problems of control theory has been a subject of { study} for a long time, see \cite{Bini2012,Lancaster95,Huang2018} and the references therein.

In this paper, we are concerned with the discrete-time algebraic Riccati equation (DARE)
\begin{align}\label{eq:NMEP}
X=R_{\rm d}(X):= H+A^\ast X (I+ GX)^{-1} A,
\end{align}
where $A\in \mathbb{C}^{n \times n}$, $G$ and $H$ are positive semidefinite matrices of size $n\times n$, and the $n$-square matrix $X$ is the unknown Hermitian matrix that is to be determined. {Here $\ast$ stands the complex conjugate transpose}. We provide new theoretical supports to the stability properties of solutions to the DARE \eqref{eq:NMEP} and new convergence results for an iteration method, proposed recently in \cite{chiang18,chiang20}.

In the past few decades, there have developed fruitful theoretical results as well as a variety of numerical algorithms for the DARE such
as the famous direct method, the Schur method, and some iterative methods including matrix disk function method, matrix sign function method, Newton iterations (NM) and structure-preserving doubling algorithm (SDA), as well as those developed in \cite{Bini2012,Lancaster95,Huang2018}. Among those methods, NM and SDA, both of which converge quadratically, are well-known for their fast convergence behaviour. In contrast to the NM, which requires the computation of the inverse of Fr$\rm\acute{e}$chet derivatives, two kinds of SDA have attracted much interests because of its nice numerical behavior, quadratic convergence rates, low computational costs, and high numerical reliability. See e.g.~\cite{Huang2018}.

Recently, the semigroup property for some binary matrix operations is investigated in \cite{chiang20} and has been applied to the construction of iterations for solving DARE. More precisely, applying the semigroup property to a fixed-point iteration $\mathbb{X}_{k+1}=F(\mathbb{X}_k,\mathbb{X}_1)$, one can obtain an accelerated iteration(AFPI) with at least the R-convergence rate of any desired order $r$. Moreover, this iterative method can be reduced to SDA when $r=2$.

 A Hermitian solution $X$ of Eq.~\eqref{eq:NMEP} is called stabilizing (respectively, almost stabilizing) if all the eigenvalues of the closed-loop matrix $T_{X}:=(I+ GX)^{-1} A$ are in the open (respectively, closed) unit disk.
 The stability property plays important role in the analysis of the convergence behaviour of SDA. In the previous works, the convergence analysis of SDA are concerned with ``the regular case'', that is, $\rho(T_{X})<1$ and ``the critical case'' when $\rho(T_{X})=1$. It has been proved that the SDA has quadratic convergence in the regular case, while for the critical case, the convergence speed is usually linear under some additional assumptions \cite{chiangSIMAX,Chiang2010,Guo1999,WWL2006, Bini2012}. For example, in the critical case,  it was proved in \cite{chiangSIMAX} that the SDA converges linearly to an { almost stabilizing} solution $X_+$ with rate at least $1/2$ if each unimodular eigenvalue of $T_{X_+}$ has a half of the partial multiplicity of  $\mathcal{M}-\lambda \mathcal{L}$, where \begin{align*}
\mathcal{M}-\lambda \mathcal{L}=\bb A & 0 \\ -H & I \eb-\lambda \mathcal{L} \bb I & -G \\ 0 & A^\ast \eb
 \end{align*}
is a matrix pencil associated with the Eq.~\eqref{eq:NMEP}. Note that $\mathcal{M}\bb I \\ X\eb=\mathcal{L}\bb I \\ X\eb T_{X}$ if and only if $X$ is a Hermitian solution of Eq.~\eqref{eq:NMEP}. One contribution in this paper is shown that the SDA still converges quadratically when solving equations of type \eqref{eq:NMEP} even in the critical case.

As mentioned above the convergence assumptions in the critical case are highly related to the distribution of the partial multiplicity of the unimodular eigenvalue of $\mathcal{M}-\lambda \mathcal{L}$, which is very sensitive to perturbations and difficult to compute. Therefore, the convergence criterion is not easy to discriminate due to the ill-posed problem, which must be solved at this end. In the practical implementations it is expected to avoid computing the elementary divisors corresponding to the eigenvalues of $\mathcal{M}-\lambda \mathcal{L}$ on the unit circle. By the way, we are interested in the case where some eigenvalues of $\mathcal{M}-\lambda \mathcal{L}$ are outside the unit circle and we study the convergence behaviors of the SDA when $\rho(T_{X})>1$.

 The main contribution of this paper is to show that, under ceratin conditions, the quadratic convergence of AFPI still holds in the critical case and even $\rho(T_{X_+})>1$. we show that the assumption on the partial multiplicities of unimodular eigenvalue can be reduced to any size, which has not been discussed in the previous works.
This paper is organized as follows. In Section 2, we provide some preliminaries used in solving matrix equations; In Section~3, we describe how to use a congruent transformation in order to reduce the compact structure of Eq.~\eqref{eq:NMEP} and provide a fixed-point iteration with R-linearly convergence to compute the minimal positive definite solution, including but not limited to $\rho(T_X)<1$; An R-superlinearly convergent iterative method with order $r>1$ is discussed and two numerical experiments show that the
reliability and efficiency of the proposed method in Section 4.
  Finally, concluding remarks are given in Section 6.

In the subsequent discussion, { the symbols $\mathbb{C}^{n \times n}$, $\mathbb{H}_{n}$, $\mathbb{N}_{n}$ and $\mathbb{P}_{n}$ stand for the set of $n\times n$ complex matrices, Hermitian matrices, positive semidefinite matrices and positive definite matrices, respectively.} We denote, the open unit disk by $\mathbb{D}$, the closed unit disk by $\bar{\mathbb{D}}$ {and the unit disk by $\mbox{bd}(\mathbb{D})$}, the $m\times m$ identity matrix by $I_m$, the conjugate matrix of $A$ by $\overline{A}$, the conjugate transpose matrix of $A$ by $A^\ast$, the spectrum of $A$ by $\sigma(A)$ and use $\rho(A)$ to denote the spectral radius of a square matrix $A$, { and $\mbox{nullity}(A)$ stands for the dimension of the kernel of a matrix $A$.} We use the symbol $A> 0$ (or $A\geq 0$) to represent that $A$ is a Hermitian positive definite matrix (or a Hermitian positive semidefinite matrix) and {the Loewner order} $A > B$ (or $A\geq B$ ) with two Hermitian matrices $A$ and $B$  if $A - B > 0$ (or $A-B\geq 0$). { We use $[A,B]$ to denote the subset $\{C\in \mathbb H_n| A\leq C\leq B\}$ and use $A\oplus B$ to denote the direct sum of two square matrices $A$ and $B$}. A Hermitian solution $X_{+}$ of Eq.\eqref{eq:NMEP} is called maximal (or minimal) if $X_{+}\geq S$ (or $X_{+}\leq S$) for every Hermitian solution $S$.

\section{Preliminaries}\label{sec:SP}
We start this section by recalling some useful definitions and theoretical results concerning Eq.\eqref{eq:NMEP}.
As one of the most important evaluations of an iterative method,  the speed of convergence is usually measured by the R-order convergence, which is introduced in \cite{Ortega2000,Kelley95,Potra1989,Bini2012}.
\begin{Definition}\label{def:cov}
Given a sequence $\{X_k\}\subseteq\mathbb{C}^{n\times n}$ and an induced matrix norm $\|.\|$, then
 $X_k$ converges \rm{R}-linearly to $X_\star$ if
\begin{align*}
\limsup\limits_{k\rightarrow\infty}\sqrt[k]{\|X_k -  X_\star\|} \leq \sigma, \quad \sigma \in (0,1),
\end{align*}
and $X_k$ converges \rm{R}-superlinearly to $X_\star$ with at least order $r$ if
\begin{align*}
\limsup\limits_{k\rightarrow\infty}\sqrt[r^k]{\|X_k -  X_\star\|}\leq\sigma, \quad \sigma \in (0,1),
\end{align*}
{
where $r>1$ is an integer.
}
\end{Definition}

The following result is a generalization of the completeness of real number.

\begin{Lemma}\label{MCT}~\cite{Bernstein2009}[Proposition 8.6.3]
Let $\{A_i\}_{i=1}^\infty$ be a sequence of positive semidefinite matrices satisfying $A_j\geq A_i \geq 0$ if  $j\geq i$, and assume that $B$ is another positive semidefinite matrix satisfying
$B \geq A_i$ for all $i > 0$. Then,
$A=\lim\limits_{i\rightarrow \infty} A_i$ exists and
$B\geq A \geq 0$.
\end{Lemma}
 A matrix operator $f:\mathbb{H}_{n}\rightarrow\mathbb{H}_{n}$ is order preserving (resp. reversing) if $f(A)\geq f(B)$ (resp. $f(A)\leq f(B)$) when $A\geq B$ and $A,B\in\mathbb{H}_{n}$. The first proposition is vital and elementary.

{
\begin{Proposition}\label{pro1}
Under the assumptions on the coefficients $A$, $G$, and $H$,
the matrix operator $R_{\rm d}:\mathbb{N}_n\rightarrow\mathbb{N}_n$ is order preserving on $\mathbb{N}_n$.
\end{Proposition}
\begin{proof}
We prove the result by showing that
\begin{align}\label{1}
(I_n+AC)^{-1}A\geq(I_n+BC)^{-1}B
\end{align}
if $A\geq B$ for any positive semidefinite matrices $A$,$B$ and $C$ of size $n$.
Let $A_\epsilon:=A+\epsilon I_n$ and $B_\epsilon:=B+\epsilon I_n$ for $\epsilon>0$. It is evident that
\begin{align}\label{2}
(I_n+A_\epsilon C)^{-1}A_\epsilon=(A_\epsilon^{-1}+C)^{-1}\geq(B_\epsilon^{-1}+C)^{-1}=(I_n+B_\epsilon C)^{-1})B_\epsilon,
\end{align}
Take the limit as $\epsilon$ goes to zero from the right on both sides of \eqref{2}, which promptly yields \eqref{1} by continuity argument.

This proposition immediately follows from the below inequality,
\[
R_{\rm d}(X)-R_{\rm d}(Y)=A^\ast[(I_n+XG)^{-1}X-(I_n+YG)^{-1}Y]A\geq 0,
\]
if $X\geq Y$.
\end{proof}
}
Consider the matrix equation
\begin{align}\label{me1}
X=F(X),
\end{align}
where $F$ is a monotone matrix operator on $\mathbb{H}_n$. The following theorem shows the existence of extreme Hermitian solutions of Eq.~\eqref{me1}.
\begin{Theorem}\label{pre}
{
Assume that the matrix operator $F:\mathbb{H}_n\rightarrow\mathbb{H}_n$ is order preserving.} Let
$ S_1:=\{X\in\mathbb{H}_n|X\geq F(X)\}$ and
  $S_2:=\{X\in\mathbb{H}_n|X\leq F(X)\}.$ Consider the fixed-point iteration  defined by
\begin{align*}
X_{k+1}=F(X_k),
\end{align*}
with an initial matrix $X_1\in\mathbb{H}_n$.
Suppose that there is $\widehat{X}_1\in S_1$ and $\widehat{X}_2\in S_2$ such that $\widehat{X}_1\geq \widehat{X}_2$. Then, we have the following statements:
\begin{itemize}
\item[{\rm (1a).}] $F([\widehat{X}_2,\widehat{X}_1])\subseteq [\widehat{X}_2,\widehat{X}_1]$ and there is a $X\in [\widehat{X}_2,\widehat{X}_1]$ solving Eq.~\eqref{me1} if $F$ is continuous on $[\widehat{X}_2,\widehat{X}_1]$.
\item[{\rm (1b)}.] The sequence $\{X_k^-\}$ with $X_1^-=\widehat{X}_2$ is monotonically nondecreasing and converges to a solution $X_\star^-$ of Eq.~\eqref{me1} and $X_\star^-\leq \widehat{X}_1$.
  \item[{\rm (1c)}.] The sequence $\{X_k^+\}$ with $X_1^+=\widehat{X}_1$ is monotonically nonincreasing and converges to a solution $X_\star^+$ of Eq.~\eqref{me1} and $X_\star^+\geq \widehat{X}_2$.

  \item[{\rm (1d)}.] For any positive integer $k$, we have
\[
\widehat{X}_2 \leq X_k^-\leq  X_{k+1}^- \leq  X_\star^- \leq X_\star^+\leq X_{k+1}^+\leq  X_{k}^+\leq \widehat{X}_1.
\]
Furthermore, $X_\star^+$ is the maximal Hermitian solution of \eqref{me1} if $\widehat{X}_1$ is an upper bound of $S_2$, and $X_\star^-$ is the minimal Hermitian solution of \eqref{me1} if $\widehat{X}_2$ is a lower bound of $S_1$.

\end{itemize}
%
\end{Theorem}
\begin{proof}
Concerning part (1a), since $F$ is order preserving, it can be easily seen that {$F(X)\in[F(\widehat{X}_2),F(\widehat{X}_1)]\subseteq [\widehat{X}_2,\widehat{X}_1]$} for any {$X\in[\widehat{X}_2,\widehat{X}_1]$}. If $F$ is continuous, then $F$ has a fixed point in {$[\widehat{X}_2,\widehat{X}_1]$} from the Brouwer fixed point theorem.

Parts (1b) and (1c) can be proved easily by an induction on $k$. Concerning part (1b), we have $\widehat{X}_1-X_1^-=\widehat{X}_1-\widehat{X}_2\geq 0$ and ${X}_2^- -X_1^-=F(\widehat{X}_2)-\widehat{X}_2\geq 0$. From the inductive assumption
  $\widehat{X}_1-X_k^-\geq 0$ and ${X}_{k+1}^- -X_{k}^-\geq 0$,  we derive the following inequalities
\begin{align*}
  \widehat{X}_1-X_{k+1}^-&\geq F(\widehat{X}_1)-F(X_k^-)\geq 0, \\
  {X}_{k+2}^- -X_{k+1}^- &= F({X}_{k+1}^-)-F(X_{k}^-)\geq 0,
\end{align*}
which completes the induction process. Part (1c) can be proved analogously.

For the inequalities in part (1d), we only need to prove $X_\star^-\leq X_\star^+$ since the rest is a direct consequence of parts (1b) and (1c). Observe that $\widehat{X}_2\leq \widehat{X}_1$ so that $F^{(k)}(\widehat{X}_2)\leq F^{(k)}(\widehat{X}_1)$, which implies $X_\star^-=\lim\limits_{k\rightarrow\infty} F^{(k)}(\widehat{X}_2)\leq \lim\limits_{k\rightarrow\infty} F^{(k)}(\widehat{X}_1)= X_\star^+$.

For the rest of the statement, it is easily seen that $S_2$ contains all Hermitian solution of Eq.~\eqref{me1}.
{ Let $X\in S_2$. Observe that $X_1^+=\widehat{X}_1\geq X$ and
$X_{k+1}^+=F(X_{k}^+)\geq F(X)\geq X$ if $X_{k}^+\geq X$ for any integer $k\geq 1$. One may easily prove by induction that $X_{k}^+\geq X$ for all integer $k\geq 1$. Thus, $X_{\star}^+=\lim\limits_{k\rightarrow\infty}X_{k}^+\geq X$ with all Hermitian solution $X$ of Eq.~\eqref{me1}.
}
 The same argument is applied to the proof of the minimal Hermitian solution $X_\star^-$.

%
\end{proof}
As previously mentioned,  the convergence behavior of  SDA for solving DARE \eqref{eq:NMEP} is highly related to the distance between the unit circle and the spectral radius of the closed-loop matrix, which is characterized by the stability properties of the solution of equation \eqref{eq:NMEP}.

A useful tool in the estimation of $\rho(T_{X})$ is the inertia theorems for Stein matrix equation (SME). In stability theory and its applications many results for SME has attracted great attentions for its theoretical and practical significance in systems control~\cite{Lancaster95}. Let the Stein matrix operator $S_A:\mathbb{N}_n\rightarrow\mathbb{H}_n$ associated with a matrix $A\in\mathbb{C}^{n\times n}$ be defined by
\begin{align}\label{SME}
S_A(X):= X-A^\ast X A,
\end{align}
for any $X\in\mathbb{N}_n$. In general, the operator $S_A$ is neither order preserving nor order reversing. However, under the assumption that $\rho(A)<1$
the operator $S^{-1}_A:\mathbb{N}_n\rightarrow\mathbb{N}_n$ exists, and $S^{-1}_A$ is order preserving since $S_A^{-1}(X)=\sum\limits_{k=0}^\infty (A^k)^\ast X A^k \geq \sum\limits_{k=0}^\infty (A^k)^\ast Y A^k=S_A^{-1}(Y)$ for $X\geq Y$.

 In the rest of the section, we propose the stability properties of a discrete-time Lyapunov operator and equation. We begin with the definitions of Lyapunov stability and asymptotic stability of a matrix.
\begin{Definition}
\cite[Definition 11.8.1.]{Bernstein2009}
Let $A\in\mathbb{C}^{n\times n}$. The notation ${\rm ind}_\lambda(A)$  denotes the size of the largest Jordan block corresponding the eigenvalue $\lambda$ of $A$. Furthermore,
\begin{enumerate}
   \item A is discrete-time asymptotically stable if $\rho(A)< 1$.
  \item A is discrete-time Lyapunov stable if $\rho(A)\leq 1$ and ${\rm ind}_\lambda(A)=1$ for all unimodular eigenvalues of $A$.
\end{enumerate}
\end{Definition}

In analogy with~\cite{Bernstein2009}[Proposition 11.10.5], we have the following similar result, the proof can be found in the Appendix.
\begin{Lemma}\label{op}
Let $S_A:\mathbb{N}_n\rightarrow\mathbb{H}_n$ be the Stein matrix operator defined in \eqref{SME}. Then, we have the following statements:
\begin{enumerate}
 \item there exists a matrix $X_0\in\mathbb{P}_n$ such that $S_A(X_0)> 0$  if and only if $A$ is discrete-time asymptotically stable.
  \item there exists a matrix $X_0\in\mathbb{P}_n$ such that $S_A(X_0)\geq 0$ if and only if $A$ is discrete-time Lyapunov stable. Furthermore, the numbers of unimodular eigenvalues (counting multiplicities) is the {nullity of $S_A(X_0)$}.
\end{enumerate}
\end{Lemma}
As a consequence of Lemma \ref{op}, we have
\begin{Corollary}\label{DLMEP}
For a matrix $A\in \mathbb{C}^{n\times n}$, $A$ is discrete-time asymptotically (or Lyapunov) stable if exactly one of the following possibilities occurs.
\begin{subequations}
\begin{enumerate}
  \item Let the Stein matrix equation with sign ``-'' be defined by
\begin{align}\label{CS11}
X-A^\ast X A=Q,
\end{align}
where $A\in\mathbb{C}^{n\times n}$ and $Q>0$ (or $Q\geq0$). Assume that there exists a solution $X\in\mathbb{P}_n$ solving \eqref{CS11}.
  \item Let the Stein matrix equation with sign ``+'' be defined by
\begin{align}\label{CS12}
X+A^\ast X A=Q,
\end{align}
where $A\in\mathbb{C}^{n\times n}$ and $Q>0$. Assume that $Q> A^\ast Q A$ (or $Q\geq A^\ast Q A$) and there exists a solution $X\in\mathbb{P}_n$ solving \eqref{CS12}.
\end{enumerate}
\end{subequations}
\end{Corollary}
\begin{proof}
{Part 1 is a direct consequence of} Lemma~\ref{op}. Concerning part 2, let $X>0$ be a solution of Eq.~\eqref{CS12}, { then} $X-(A^\ast)^2 X A^2=Q-A^\ast Q A {>0 \ (\geq 0)}$, {the proof is completed by applying Lemma~\ref{op}.}
\end{proof}
The following simple result is useful, which is used to obtain the main result.
\begin{Lemma}\label{op1}
Let $J$ be the Jordan canonical matrix with size $n\times n$. If $\rho(J)\geq 1$ and  $S_J(X_{\rm p})\geq 0$ for some $X_{\rm p}\geq 0$, then, we have the following statements: 
\begin{enumerate}
  \item $S_J(X_{\rm p})=0$. In other words, $S_J(N_n)\cap N_n=0_n$.
  \item $X_{\rm p}=0$ if $\rho(J)>1$ and $X_{\rm p}=xe_ne_n^\top$ if $\rho(J)=1$, where $x\geq 0$.
\end{enumerate}
\end{Lemma}
\begin{proof}
For the sake of simplicity, the notation $[A]_{i,j}$ denotes $(i,j)$ entry of a matrix $A\in\mathbb{C}^{m\times m}$ for $1\leq i,j\leq m$.
For any positive semidefinite matrix $A$,
it is clear that some diagonal entry $[A]_{i,i}=0$ if and only if the row and the column containing $[A]_{i,i}$ consist entirely of 0.

According to the assumption, we write $J=aI_n+N$, where $|a|\geq 1$ and $N$ is a $n\times n$ nilpotent matrix. Let $Y_{\rm p}=S_J(X_{\rm p})=X_{\rm p}-J^\ast X_{\rm p} J$. In the case of $|a|>1$, we have $0\leq [Y_{\rm p}]_{1,1} = [X_{\rm p}]_{1,1}- |a|^2e_1^\top X_{\rm p} e_1\leq 0$ and thus $0=[Y_{\rm p}]_{1,1}=[X_{\rm p}]_{1,1}$. If $|a|=1$, it is immediate that $[Y_{\rm p}]_{1,1} = [X_{\rm p}]_{1,1}- e_1^\top X_{\rm p}e_1=0$. {Since $Y_{\rm p}\geq 0$, we have $[Y_{p}]_{1,j}=[Y_{\rm p}]_{j,1}=0$ for $j=2,\ldots,n$.}  so that, $0=[Y_{\rm p}]_{1,2} =[X_{\rm p}]_{1,2}-(\bar{a} e_1^\top) X_{\rm p} (e_1+a e_2)=-\bar{a}[X_{\rm p}]_{1,1}$. It follows that $[X_{\rm p}]_{1,1}=0$ and thus $[X_{\rm p}]_{1,j}=[X_{\rm p}]_{j,1}=0$ for $1\leq j\leq n$.

For $2\leq k \leq n-1$, we observe that
{
\begin{subequations}\label{yy}
 \begin{align}
  [Y_{\rm p}]_{k,k+1} &=(1-|a|^2)[X_{\rm p}]_{k,k+1}-([X_{\rm p}]_{k-1,k}+\bar{a} [X_{\rm p}]_{k,k}+ a [X_{\rm p}]_{k-1,k+1}),\label{yy1}\\
  [Y_{\rm p}]_{k,k}  &=(1-|a|^2)[X_{\rm p}]_{k,k}-([X_{\rm p}]_{k-1,k-1}+2\mbox{Re}(a [X_{\rm p}]_{k-1,k})).\label{yy2}
\end{align}
\end{subequations}
}
{ Note that \eqref{yy2} also holds for $k=n$.}
Let $k=2$ and $|a|=1$. Substituting $[X_{\rm p}]_{1,1}=[X_{\rm p}]_{1,2}=0$ into \eqref{yy2} yields $[Y_{\rm p}]_{2,2}=0$ {so that $[Y_p]_{2,j}=[Y_p]_{j,2}=0$ for $1\leq j\leq n$. Substituting} $[Y_{\rm p}]_{2,3}=[X_{\rm p}]_{1,2}=[X_{\rm p}]_{1,3}=0$ into \eqref{yy1} yields $[X_{\rm p}]_{2,2}=0$. Repeat this procedure we have $[X_{\rm p}]_{k,k}=[Y_{\rm p}]_{k,k}=0$ for $1\leq k \leq n-1$. Hence, $X_{\rm p}=0\oplus [X_{\rm p}]_{n,n}\geq 0$ and $Y_{\rm p}=0$.

On the other hand, let $k=2$ and $|a|>1$. Substituting $[X_{\rm p}]_{1,1}=[X_{\rm p}]_{1,2}=0$ into \eqref{yy2} we have $0\leq[Y_{\rm p}]_{2,2}=(1-|a|^2)[X_{\rm p}]_{2,2}$ {so that} $[Y_{\rm p}]_{2,2}=[X_{\rm p}]_{2,2}=0$. {Repeating the procedure on $k$, we obtain} $[Y_{\rm p}]_{k,k}=[X_{\rm p}]_{k,k}=0$ for $k>2$. Hence, $X_{\rm p}=Y_{\rm p}=0$.
\end{proof}
\section{Main results on the DARE}
{
To make our main results more clearly and explicitly, the
rest of the section is divided into two parts, respectively: One is the sufficient condition for the existence of extreme solutions of DARE and the other is new convergence results to a fixed point iteration $X_{k+1}=R_{\rm d}(X_k)$.
}
\subsection{New results on the extreme solutions of DARE}
First of all, we consider the extreme solutions of DARE~\eqref{eq:NMEP}. Inspired by the results of Theorem~\ref{pre}, we introduce two subsets on $\mathbb{N}_n$:
\begin{align*}
  \mathbb{R}_{\geq}:=\{X\in\mathbb{N}_n|X\geq R_{\rm d}(X)\},\,
  \mathbb{R}_{\leq}:=\{X\in\mathbb{N}_n|X\leq R_{\rm d}(X)\}.
\end{align*}
Then we have the following result concerning the existence of minimal and maximal positive semidefinite solutions of Eq.~\eqref{eq:NMEP}.
\begin{Lemma}\label{lem-sayed}
Let $\mathbb{S}_{\geq}:=\{X\in\mathbb{N}_n|S_A(X)\geq H\}$ and $\mathbb{S}_{\leq}:=\{X\in\mathbb{N}_n|S_A(X)\leq H\}$. Consider the fixed-point iteration $X_{k+1}=R_{\rm d}(X_k)$ with an initial $X_1$. Then we have the following statements:
\begin{enumerate}
  \item Assume that $ \mathbb{R}_{\geq}\neq \phi$ and let $X_1\in[0,H]$. Then, the sequence $\{X_k\}$ converges increasingly to the minimal positive semidefinite solution of Eq.~\eqref{eq:NMEP}.
  \item Assume that $\mathbb{S}_{\geq}\neq\phi$ and let $X_1\in\mathbb{S}_{\geq}$, then the sequence $\{X_k\}$ converges decreasingly to a positive semidefinite solution of Eq.~\eqref{eq:NMEP}.
  \item Assume that $\rho(A)<1$. Then, $\mathbb{S}_{\geq}\neq\phi$ and the sequence $\{X_k\}$ with $X_1\in\mathbb{S}_{\geq}$ converges decreasingly to the maximal positive semidefinite solution of Eq.~\eqref{eq:NMEP}.
\end{enumerate}
\end{Lemma}
\begin{proof}
Applying the Woodbury identity, we have $X(I+GX)^{-1}=X-XG(I+XG)^{-1}X$, it is then easily seen that $X-R_{\rm d}(X)=S_A(X)-H+A^\ast XG(I+XG)^{-1}XA$,
from which we know $\mathbb{S}_{\geq}\subseteq\mathbb{R}_{\geq}$ and $\mathbb{R}_{\leq}\subseteq\mathbb{S}_{\leq}$.

For part 1, {observe that $X_1\in \mathbb{R}_{\leq}$ and $X_1\leq H\leq X$ for any $X\in \mathbb{R}_{\geq}$, this implies that $X_1\in \mathbb{R}_{\leq}$ is a lower bound of $\mathbb{R}_{\geq}$. It follows from Theorem ~\ref{pre} that the sequence $\{X_k\}$ is monotonically increasing and converges to the minimal positive semidefinite solution of Eq. \eqref{eq:NMEP}.}

 Concerning part 2, it follows from $X_1\in \mathbb{S}_{\geq}$ that  $X_1\in \mathbb{R}_{\geq}$, so that $X_1\geq H$. Analogous to Theorem~\ref{pre}, one can prove that the sequence $\{X_k\}$ is monotonically decreasing and converges to a positive definite solution of Eq. \eqref{eq:NMEP}.

   For part 3, {observe that for any $X\in \mathbb{R}_{\leq}$ we have $X\in \mathbb{S}_{\leq}$, which yields $S_A(X_1)\geq H\geq S_A(X)$, the first inequality holds since {\color{red}$X_1\in \mathbb{S}_{\geq}$}. Hence,} $X_1\geq X$ since $\rho(A)<1$. {That is, $X_1\in \mathbb{R}_{\geq}$ is an upper bound of $\mathbb{R}_{\leq}$, which, together with part (1d) of Theorem \ref{pre}, shows that the sequence $\{X_k\}$ with $X_1\in \mathbb{S}_{\geq}$ converges decreasingly to the maximal positive semidefinite solution of Eq. \eqref{eq:NMEP}.}
\end{proof}
We also notice that the proof of the final part of Lemma~\ref{lem-sayed} was motivated by \cite{Sayed01}[Theorem 5.1]. It seems that the assumption $\mathbb{R}_{\geq}\neq\phi$ is not easy to check. An useful sufficient condition on the coefficient matrices for the existence of the positive semidefinite solution of Eq.\eqref{eq:NMEP} can be written as follows.

\begin{Corollary}\label{thm:pd1}
Assume that the matrices $A\in\mathbb{C}^{n\times n}$, $G\in\mathbb{N}_n$ and $H\in\mathbb{N}_n$ satisfy one of the following two conditions:
\begin{enumerate}
  \item $G$ is nonsingular, i.e., $G>0$.
  \item $G$ is singular and  $\rho(A^\ast A)\leq 1$.
\end{enumerate}
Then, there exists a positive semidefinite solution to~\eqref{eq:NMEP} and thus $\mathbb{R}_{\geq}\neq\phi$.
\end{Corollary}
\begin{proof}
Let $h=\max\limits_{\lambda\in\sigma(H)}\lambda=\rho(H)$, $g=\min\limits_{\lambda\in\sigma(G)}\lambda$ and $a=\max\limits_{\lambda\in\sigma(A^\ast A)}\lambda=\|A\|_2^2$. It can be shown that the quadratic inequality
\[
g x^2+(1-a-hg)x-h\geq 0
\]
has a nonnegative solution $x_c$ if $g\neq 0$ or $g=0$ and $a<1$. In each
assumption we have
\[
x_c I_n \geq (h+\frac{ax_c}{1+gx_c})I_n \geq R_{\rm d}(x_c I_n).
\]
Thus, $\mathbb{R}_\geq$ is nonempty. From Theorem~\ref{solution}, there exists a positive semidefinite solution of~\eqref{eq:NMEP}.

\end{proof}

{
By the way, we are concerned with the (almost) stabilizing solution based on the following observation. Let $X\in \mathbb{R}_{\geq}$, then we have
\begin{align*}
  X-(T_X)^\ast X T_X&\geq H+(T_X)^\ast (X+XGX)T_X-(T_X)^\ast X T_X\\
  &= H+(T_X)^\ast XGX T_X\geq 0.
\end{align*}
}
In view of part 2 of Lemma~\ref{op}, we know that $\rho(T_X)\leq 1$ and ${\rm ind}_\lambda(T_X)=1$ for all $\lambda\in\sigma(A)\cap\mbox{bd}(\mathbb{D})$ if and only if $X>0$. Moreover, $\rho(T_X)<1$ if $H>0$. The above conclusion is summarized as follows.
\begin{Theorem}
Under the typical assumption that $H\geq0$ and $G\geq 0$, any positive definite solution $X$ of \eqref{eq:NMEP} is an almost stabilizing solution. Furthermore, $X$ is a stabilizing solution if $H>0$.
\end{Theorem}

\subsection{New results on the convergence of a fixed point iteration}
{
In this subsection, combine Lemma~\ref{op1}, and Lemma~\ref{lem3}, we will now examine the convergence behaviour of the fixed-point iteration $X_{k+1}=R_{\rm d}(X_k)$ with $X_1=H$. Assume that $\mathbb{R}_{\geq}\neq\phi$. From Lemma~\ref{lem-sayed} we know that $\{X_k\}$ converges to the minimal positive semidefinite solution $X_\star$.

Our main result includes the convergence in the case where $\rho(T_{X_\star})\geq 1$, so that generalizes the previous result in \cite{Huang2018}, where only the convergence in the case where $\rho(T_{X_\star})<1$ is considered. }The concepts of convergence speed of fixed-point iteration for solving DARE~\eqref{eq:NMEP} are related by the following result; see e.g., \cite{chiang18}[Appendix].
{
\begin{Lemma}\label{lem4}
 Let $Z_{k+1}=R_{\rm d}(Z_{k})$ be the fixed point iteration of~\eqref{eq:NMEP} with an initial matrix $Z_1$. If $Z_k$  converges to $Z_\star$, then
\begin{equation*}
\limsup\limits_{k\rightarrow\infty}\sqrt[k]{\|{Z}_\star-{Z}_k\|}\leq \rho(T_{Z_\star})^2.
\end{equation*}
\end{Lemma}
}
Lemma \ref{lem4} shows that the fixed-point iteration works well if $\rho(T_{Z_\star})<1$, while the efficiency for the case where  $\rho(T_{Z_\star})\geq 1$ is difficult to tell. The aims of this subsection is to verity the R-linear convergence of $\{X_k\}$ under the case that $\rho(T_{X_\star})<1$.

Before proceeding with the main result of this section, we also require the following lemma, which is the original DARE divided into three DAREs with small scale.

%
%

\begin{Lemma}~\label{lem3}
Assume that there exists a matrix $X_+\in \mathbb{R}_{\geq}$. {Let $J_{T_{X_+}}=P^{-1}T_{X_+}P$ be the Jordan canonical form of $T_{X_+}$. Suppose that $J_{T_{X_+}}= J_1\oplus J_u \oplus J_s$, where $J_1\in\mathbb{C}^{m_1\times m_1}$ with $\sigma(J_1)\subseteq{\rm bd}(\mathbb{D})$, $J_u\in\mathbb{C}^{m_2\times m_2}$ with $\sigma(J_u)\cap\bar{\mathbb{D}}=\phi$ and $J_s\in\mathbb{C}^{m_3\times m_3}$ with $\sigma(J_s)\subseteq\mathbb{D}$}.
Namely, $|\lambda|=1$ for all $\lambda\in\sigma(J_1)$, $\rho(J_u^{-1})< 1$ and $\rho(J_s)<1$.
Then, we have the following statements:
\begin{itemize}
  \item [{\rm (a)}] We have nullity$(H)\geq m_1+m_2$. Furthermore, $H$ is congruent to a block diagonal matrix $0_{m_1+m_2} \oplus \widehat{H}_{s}$, where $\widehat{H}_{s}\in \mathbb{C}^{m_3\times m_3}$.
  \item [{\rm (b)}] Suppose that $X_+=R_{\rm d}(X_+)$. Then, $X_+$ is congruent to the block diagonal matrix $X_{+,1}\oplus X_{+,u}\oplus X_{+,s}$, where $X_{+,1}$, $X_{+,u}$ and $X_{+,s}$ are respectively positive semidefinite solution of DARE: $X=R_{\rm d,1}(X)$, $X=R_{\rm d,u}(X)$ and $X=R_{\rm d,s}(X)$. Furthermore, $T_{X_{+,1}}=J_1$, $T_{X_{+,u}}=J_u$ and $T_{X_{+,s}}=J_s$.

\end{itemize}

\end{Lemma}
\begin{proof}
   Let $Y_+=P^\ast X_+P$, $\widehat{A}=P^{-1}AP$, {\color{red} $\widehat{G}=P^{-1}GP^{-\ast}\geq 0$}, and $\widehat{H}=P^\ast H P\geq 0$.
Note that $n=m_1+m_2+m_3$. For the sake of convenience, we partition any $\mathcal {A}\in \mathbb{C}^{n\times n}$ as a $3 \times 3$ block matrix $\mathcal{A}=[ [\mathcal A]_{i,j}]$, where $[\mathcal{A}]_{i,j}\in \mathbb{C}^{m_i\times m_j}$ and $1\leq i,j \leq 3$. The notation $[\mathcal{A}]_{[i_1,i_2]\times[j_1,j_2]}$ denotes the block submatrix of $\mathcal{A}$ consisting of the $i_1,i_1+1,\ldots,i_2$ rows and the $j_1,j_1+1,\ldots,j_2$ columns. We use $[\mathcal{A}]_{[i_1,i_2]}$ for short if $i_1=j_1$ and $i_2=j_2$. As a consequence, $[\mathcal{A}]_{i,i}=[\mathcal{A}]_{[i,i]}\geq 0$  {\color{red}if $\mathcal{A} \geq 0$ for $i=1,2,3$.}

{Concerning part (a), the} inequality $X_+\geq R_{\rm d}(X_+)$ is equivalent to the inequality
\begin{align}\label{eq:block}
Y_+ \geq \widehat{H}+J_{T_{X_+}}^\ast Y_+J_{T_{X_+}}+J_{T_{X_+}}^\ast \widehat{G}_{Y_+} J_{T_{X_+}},
\end{align}
where $\widehat{G}_{Y_+}{=} Y_+\widehat{G}Y_+$.  A direct computation of the upper left corner $(m_1+m_2)\times(m_1+m_2)$ block of \eqref{eq:block} yields
\begin{align*}
&[Y_+]_{[1,2]}-(J_{1}\oplus J_{u})^\ast[Y_+]_{[1,2]}(J_{1}\oplus J_{u})\\
&\geq [\widehat{H}]_{[1,2]}+(J_{1}\oplus J_{u})^\ast[\widehat{G}_{Y_+}]_{[1,2]}(J_{1}\oplus J_{u})\geq 0.
\end{align*}

From Lemma~\ref{op1} we have
\begin{align}\label{Y12}
[Y_+]_{[1,2]}=(J_{1}\oplus J_{u})^\ast[Y_+]_{[1,2]}(J_{1}\oplus J_{u}),
\end{align}
{and thus $[\widehat{H}]_{[1,2]}+(J_{1}\oplus J_{u})^\ast[\widehat{G}_{Y_+}]_{[1,2]}(J_{1}\oplus J_{u})=0$, from which we deduce that} $[\widehat{H}]_{[1,2]}$ and $[\widehat{G}_{Y_+}]_{[1,2]}$ are null, and we obtian
\begin{align}\label{HG}
\widehat{H}=\bb
      0_{m_1+m_2} & 0 \\
      0 & [\widehat{H}]_{3,3} \\
    \eb,\ \widehat{G}_{Y_+}=\bb
      0_{m_1+m_2} & 0 \\
      0 & [\widehat{G}_{Y_+}]_{3,3} \\
    \eb.
\end{align}

{Concerning part (b)}, observe that the original DARE \eqref{eq:NMEP} with respect to the unknown $X$ is equivalent to the following equation with respect to the unknown $Y=P^\ast XP$,
\begin{align}\label{eq:reduced}
Y=\widehat{H}+J_{T_X}^\ast YJ_{T_X}+J_{T_X}^\ast \widehat{G}_{Y} J_{T_X}.
\end{align}
We claim that Eq.~\eqref{eq:reduced} has a positive semidefinite solution ${Y}_+$ such that ${Y}_+=[{Y}_{+}]_{1,1}\oplus[{Y}_{+}]_{2,2}\oplus[{Y}_{+}]_{3,3}$  if { there exists a $X_+\geq 0$ such that} {$R_{\rm d}(X_+)=X_+$}, where ${[Y_+]}_{1,1}\in\mathbb{N}_{m_1}$, ${[Y_+]}_{2,2}\in\mathbb{N}_{m_2}$ and ${[Y_+]}_{[3,3]}\in\mathbb{N}_{m_3}$.
Indeed, substituting \eqref{HG} into \eqref{eq:reduced} we have ${[Y_+]}_{[1]\times[2,3]}=J_1^\ast{[Y_+]}_{[1]\times[2,3]}(J_u\oplus J_s)$,
{which implies ${[Y_+]}_{[1]\times[2,3]}=({[Y_+]}_{[2,3]\times[1]})^\ast=0$ since $\bar{\lambda}\mu\neq 1$ for any $\lambda\in\sigma(J_1)$ and $\mu\in\sigma(J_u\oplus J_s)$}.
Combining these with \eqref{Y12} we can assert that $[Y_+]_{1,1}=J_1^\ast [Y_+]_{1,1} J_1$ and $[Y_+]_{2,2}=J_u^\ast [Y_+]_{2,2} J_u$. From Lemma~\ref{op1} it follows easily immediately that $[Y_+]_{2,2}=0$ since $\rho(J_u^{-1})< 1$. We conclude that $Y_+$ is a block diagonal matrix.

On the other hand, since
\begin{align*}
 0&=[\widehat{G}_{Y_+}]_{1,1}=[Y_+ \widehat{G} Y_+]_{1,1}=[(\widehat{G}^{1/2}Y_+)^\ast (\widehat{G}^{1/2}Y_+)]_{1,1}\\
 &=[(\widehat{G}^{1/2}Y_+)^\ast ]_{1,1} [\widehat{G}^{1/2}Y_+]_{1,1}+ {[\widehat{G}^{1/2}Y_+]_{[2,3]\times[1]}^\ast} [\widehat{G}^{1/2}Y_+]_{[2,3]\times[1]}\geq 0,
\end{align*}
it follows that {$[\widehat{G}^{1/2}]_{1,1}[Y_+]_{1,1}=[\widehat{G}^{1/2}Y_+]_{1,1}=0$ and $[\widehat{G}^{1/2}]_{[2,3]\times[1]}[Y_+]_{1,1}=[\widehat{G}^{1/2}Y_+]_{[2,3]\times[1]}=0$.} Thus, we see that

\begin{subequations}\label{GYS}
\begin{align}
[\widehat{G}]_{1,1}[Y_+]_{1,1}&=(([\widehat{G}^{1/2}]_{1,1})^2+[\widehat{G}^{1/2}]_{[1]\times[2,3]}[\widehat{G}^{1/2}]_{[2,3]\times[1]})[Y_+]_{1,1}=0,\label{GYS1}\\
[\widehat{G}]_{[2,3]\times[1]}[Y_+]_{1,1}&=([\widehat{G}^{1/2}]_{[2,3]\times[1]}[\widehat{G}^{1/2}]_{1,1}+[\widehat{G}^{1/2}]_{[2,3]\times[2,3]}[\widehat{G}^{1/2}]_{[2,3]\times[1]}) [Y_+]_{1,1}=0.\label{GYS2}
\end{align}
\end{subequations}
 It implies that the matrix $\widehat{G}Y_+$ can be partitioned according to the block structure
\begin{align*}
\widehat{G}Y_+&=\bb
           [\widehat{G}]_{1,1}[Y_+]_{1,1} & [\widehat{G}]_{[1]\times[2,3]}[Y_+]_{[2,3]} \\
           [\widehat{G}]_{[2,3]\times[1]}[Y_+]_{1,1} & [\widehat{G}]_{[2,3]}[Y_+]_{[2,3]} \\
        \eb\\
        &=
        \bb
           0_{m_1} & [\widehat{G}]_{[1]\times[2,3]}[Y_+]_{[2,3]} \\
           0_{(m_2+m_3)\times m_1} & [\widehat{G}]_{[2,3]}[Y_+]_{[2,3]} \\
        \eb
        =\bb
           0_{m_1} & 0_{m_1\times m_2}&[\widehat{G}]_{1,3}[Y_+]_{[3,3]} \\
           0_{m_2\times m_1}& 0_{m_2}& [\widehat{G}]_{2,3}[Y_+]_{[3,3]} \\
           0_{m_3\times m_1}& 0_{m_3\times m_2}& [\widehat{G}]_{3,3}[Y_+]_{[3,3]}
        \eb.
\end{align*}
It follows that
\begin{align*}
(I_n+\widehat{G}Y_+)^{-1}=
        \bb
           I_{m_1} & 0_{m_1\times m_2}&-[\widehat{G}]_{1,3}[Y_+]_{[3,3]}(I_{m_3}+[\widehat{G}]_{3,3}[Y_+]_{3,3})^{-1} \\
           0_{m_2\times m_1}& I_{m_2}& -[\widehat{G}]_{2,3}[Y_+]_{[3,3]}(I_{m_3}+[\widehat{G}]_{3,3}[Y_+]_{3,3})^{-1} \\
           0_{m_3\times m_1}& 0_{m_3\times m_2}& (I_{m_3}+[\widehat{G}]_{3,3}[Y_+]_{3,3})^{-1}
        \eb.
\end{align*}
{Observe that}
\begin{align}\label{Jordan}
J_{{T_{X_+}}}=J_1 \oplus J_u \oplus J_s
&=P^{-1}(I+GX_+)^{-1}AP=(I_n+\widehat{G}Y_+)^{-1}\widehat{A},
\end{align}
Compared the $(3,1)$ and $(3,2)$ positions with two sides of \eqref{Jordan}, we obtain
\begin{align*}
(I_{m_3}+[\widehat{G}]_{3,3}[Y_+]_{3,3})^{-1} [\widehat{A}]_{[3]\times[1,2]}&=0_{m_3\times (m_1+m_2)},
\end{align*}
and we deduce that $[\widehat{A}]_{[3]\times[1,2]}=0_{m_3\times (m_1+m_2)}$.
Similarly, compared the $(2,1)$ and $(1,2)$ positions with two sides of \eqref{Jordan} immediately lead to $[\widehat{A}]_{[1,2]}=0_{m_1\times m_2}$ and $[\widehat{A}]_{[2,1]}=0_{m_2\times m_1}$. That is, $\widehat{A}$ is an upper triangular block matrix.
Comparing block matrices $(1,1)$, $(2,2)$ and $(3,3)$ of two sides of \eqref{Jordan} yields
\begin{align}\label{eq:w}
J_1=[\widehat{A}]_{1,1},\,
J_{u}=[\widehat{A}]_{2,2},\,  J_{s}=(I_{m_3}+[\widehat{G}]_{3,3}[Y_+]_{3,3})^{-1}[\widehat{A}]_{3,3}.
\end{align}

Let
\begin{align*}
R_{\rm d,1}([Y_+]_{1,1})&:=[\widehat{H}]_{1,1}+([\widehat{A}]_{1,1})^\ast
[Y_+]_{1,1}(I_{m_1}+[\widehat{G}]_{1,1}[Y_+]_{1,1})^{-1}[\widehat{A}]_{1,1},\\
R_{\rm d,u}([Y_+]_{2,2})&:=[\widehat{H}]_{2,2}+([\widehat{A}]_{2,2})^\ast
[Y_+]_{2,2}(I_{m_3}+[\widehat{G}]_{2,2}[Y_+]_{2,2})^{-1}[\widehat{A}]_{2,2},\\
R_{\rm d,s}([Y_+]_{3,3})&:=[\widehat{H}]_{3,3}+([\widehat{A}]_{3,3})^\ast [Y_+]_{3,3}(I_{m_3}+[\widehat{G}]_{3,3}[Y_+]_{3,3})^{-1}[\widehat{A}]_{3,3}.
\end{align*}
Summarizing, together with $[\widehat{H}]_{[1,2]}=0_{m_1+m_2}$, $[Y_+]_{2,2}=0_{m_2}$, \eqref{GYS1}, \eqref{eq:reduced} and \eqref{eq:w} we can now formulate our main results in part~(b):
 {
 \begin{align*}
[Y_+]_{1,1}&=J_1^\ast[Y_+]_{1,1}J_1=R_{\rm d,1}([Y_+]_{1,1}),\\
[Y_+]_{2,2}&=J_u^\ast[Y_+]_{2,2}J_u=R_{\rm d,u}([Y_+]_{2,2}),\\
[Y_+]_{3,3}&=[\widehat{H}]_{3,3}+J_{s}^\ast[Y_+]_{3,3}(I_{m_3}+[\widehat{G}]_{3,3}[Y_+]_{3,3} )J_{s}
=R_{\rm d,s}([Y_+]_{3,3}),
\end{align*}
}
with $T_{[Y_+]_{1,1}}=J_1$, $T_{[Y_+]_{2,2}}=J_u$ and $T_{[Y_+]_{3,3}}=J_s$.

\end{proof}

Now, we are ready to present the main result of this subsection. The following theorem gives a sharper bound on the convergent speed of the fixed point iteration, which works both for the regular case and the critical case.
{
\begin{Theorem}\label{solution}
 Assume that $\mathbb{R}_{\geq}\neq \phi$ and $H\neq 0$. Then, the sequence $X_{k+1}=R_{\rm d}(X_k)$ with an initial matrix $X_1=H$ converges R-linearly to the minimal positive semidefinite solution $X_\star$ of \eqref{eq:NMEP}. Moreover,
the convergence rate can be shown as the following:
\[
\underset{k\rightarrow \infty}{\lim {\rm sup}} \sqrt[k]{\|X_k-X_\star\|}\leq \max\{|\lambda|^2; \lambda\in\sigma(T_{{X}_{\star}})\cap {\mathbb{D}}\}<1.
\]
Note that the $X_\star=X_k=0$ for all $k\geq 1$ and $\rho(T_{{X}_{\star}})=\rho(A)$ when $H=0$.
\end{Theorem}

\begin{proof}
In view of part 1 of Lemma \ref{lem-sayed}, the sequence $\{X_k\}$ converges to the minimal positive semidefinite solution of equation \eqref{eq:NMEP}.  It is left to prove that the convergence is R-linear. Let $T_{X_{\star}}=PJ_{T_{X_{\star}}}P^{-1}$ be the Jordan canonical decomposition of $T_{X_{\star}}=(I+GX_{\star})^{-1}A$. Set $\widehat{X}_k = P^\ast X_k P$, we have
\begin{align}
\widehat{X}_{k+1}=\widehat{H}+\widehat{A}^\ast\widehat{X}_k(I+\widehat{G}\widehat{X}_k)^{-1}\widehat{A}.
\end{align}
Together with $\widehat{X}_1 =\widehat{H}=0_{m_1+m_2}\oplus [H]_{3,3}$, we derive that the sequence $\widehat{X}_k=\widehat{X}_{k,1}\oplus\widehat{X}_{k,u}\oplus \widehat{X}_{k,s}$ is divided three sequences according to
\begin{align*}
\widehat{X}_{k+1,1}&=R_{\rm 1}(X_{k,1})=[\widehat{H}]_{1,1}+[\widehat{A}]_{1,1}^\ast\widehat{X}_{k,1}(I_{m_1}+[\widehat{G}]_{1,1}\widehat{X}_{k,1})^{-1}[\widehat{A}]_{1,1},\\
\widehat{X}_{k+1,u}&=R_{\rm u}(X_{k,u})=[\widehat{H}]_{2,2}+[\widehat{A}]_{2,2}^\ast\widehat{X}_{k,u}(I_{m_2}+[\widehat{G}]_{2,2}\widehat{X}_{k,u})^{-1}[\widehat{A}]_{2,2},\\
\widehat{X}_{k+1,s}&=R_{\rm s}(X_{k,s})=[\widehat{H}]_{3,3}+[\widehat{A}]_{3,3}^\ast\widehat{X}_{k,s}(I_{m_3}+[\widehat{G}]_{2,2}\widehat{X}_{k,s})^{-1}[\widehat{A}]_{3,3}.
\end{align*}
The conditions $\widehat{X}_{1,1}=[\widehat{H}]_{1,1}=0$ and $\widehat{X}_{1,u}=[\widehat{H}]_{1,1}=0$ implies that $\widehat{X}_{k,1}=0$ and $\widehat{X}_{k,u}=0$  for all positive integers $k$. By the way, $\widehat{X}_k=\widehat{X}_{k,1}=0$ if $H=0$.

 Let $H$ be a nonzero positive semidefinite matrix. Under the hypotheses of Lemma~\ref{lem3}, there exists a positive definite solution $\widehat{X}_{\star,s}$ of DARE $X=R_{\rm d,s}(X)$ and $\rho(T_{\widehat{X}_{\star,s}})<1$. Overall,
let $X_\star=P^{-H}(\lim\limits_{k\rightarrow \infty} \widehat{X}_k) P^{-1}=P^{-H}( 0\oplus \widehat{X}_{\star,s} )P^{-1}$. Applying Lemma~\ref{lem4} we have
\begin{align*}
\underset{k\rightarrow \infty}{\lim {\rm sup}} \sqrt[k]{\|X_k-X_\star|}&=\underset{k\rightarrow \infty}{\lim {\rm sup}} \sqrt[k]{\|P^{-H}(0\oplus\widehat{X}_{k,s}-0\oplus\widehat{X}_{\star,s})P^{-1}\|}\\& \leq\rho(T_{\widehat{X}_{\star,s}})=\max\{|\lambda|^2; \lambda\in\sigma(T_{{X}_{\star}})\cap {\mathbb{D}}\}.
\end{align*}
\end{proof}
}
\begin{Remark}
Let $n=1$, $G>0$ and $H=0$. Namely, Eq.~\eqref{eq:NMEP} has two positive semidefinite solutions $X_1=0$ with $T_{X_1}=A$ and $X_2=\frac{|A|^2-1}{G}$ with $T_{X_2}=\frac{1}{|A|}$.
In this case, $X_k=X_1=H=0$ for all $k$ and $\rho(T_{X_\star})=|A|$ can be made arbitrarily large.
\end{Remark}
The following result provides a {sufficient condition under which the sequence $\{X_k\}$ converges $R$-linearly. }
\begin{Corollary}\label{thm:pdd}
Assume that there exists a positive definite solution $X=X_+$ of \eqref{eq:NMEP}.
Then, $X_k$ converges R-linearly to the minimal positive semidefinite solution $X_\star$ of Eq.\eqref{eq:NMEP} such that
$T_{X_\star}$ is discrete-time Lyapunov stable.
\end{Corollary}
\begin{proof}
Suppose that there exists $X_+>0$ such that $X_+\in \mathbb{R}_{\geq}$.
Observe that $S_A(X_+)=H+T_{X_+}^\ast X_+GX_+T_{X_+}\geq 0$. From Lemma~\ref{op} we know that $\rho(T_{X_+})\leq 1$ and all the unimodular eigenvalues are semisimple. Consequently, $\mathbb{R}_\geq\neq\phi$ and thus $X_+$ converges to
$X_\star$ R-linearly by Theorem~\ref{solution}.
\end{proof}

\section{{An accelerated iteration} and { numerical} experiments}
In this section, for any integer $r>1$, we first show that an accelerated of fixed-point iteration (referred as AFPI) with R-superlinear convergence order $r$ is capable of computing the minimal positive semidefinite solution of equation \eqref{eq:NMEP}. It has been proved in \cite{chiang20,Chiang17} that if $\rho(T_{X_\star})<1$ the AFPI has convergence rate of any desired order $r$. We verify that the convergence speed remains invariant even $\rho(T_{X_\star})\geq 1$. It is worth mentioning that AFPI includes SDA as a special $r=2$ case~\cite{chiang20}, so that the quadratic convergence of SDA when $\rho(T_{X_\star})\geq 1$ still holds and this acts as a complementary to the existing results on the convergence of SDA.

Two numerical examples are then demonstrated to test the
accuracy of the computation and the convergence speed of AFPI under different situation; the first one show that the proposed algorithm converges suplinearly with no difficulty in the value of $\rho{(T_{X_\star})}$ less than, greater than, or equal to 1, respectively. The latter example comes from \cite{HUANG20091452} which consider a DARE~\eqref{eq:NMEP} with a negative definite matrix $H$. It is interesting to observe that our approach is still valid.
\subsection{Acceleration of fixed-point iteration}\label{Sec:acc}
The following definition characterizes the semigroup property of the iteration associated a binary operator.
\begin{Definition}~\cite{chiang20}\label{def1}
Let $D\subseteq \mathbb{C}^{n\times m}$ and $F:D\times D\rightarrow D$ be a binary matrix operator. We call that an iteration
\begin{equation}\label{eqF3}
\mathbb{X}_{k+1}=F(\mathbb{X}_k,\mathbb{X}_1), \quad k\geq 1,
\end{equation}
has the semigroup property if the operator $F$ satisfies the following associative rule:
\begin{align}\label{ar}
F(F(Y,Z),W)=F(Y,F(Z,W)),
\end{align}
for any $Y,Z$ and $W$ in $D$.
\end{Definition}
 It is interesting to point out that the sequence $\{{\mathbb{X}}_{k}\}$ satisfies the so-called discrete flow property~\cite{chiang20}[Theorem 3.2], that is,
\begin{align}\label{DF}
{\mathbb{X}}_{k+\ell}=F({\mathbb{X}}_{k},{\mathbb{X}}_{\ell}),
\end{align}
for any two positive integers $k$ and $\ell$.
%
%
Now, we construct a fixed point iteration that has the semigroup property so that the fixed point iteration can be accelerated by applying the procedure as in ~\cite{chiang20}[Algorithm 3.1]. To this end, we have
 \begin{align*}
 X = R_{\rm d}^{(k)}(R_{\rm d}(X))= R_{\rm d}^{(k+1)}(X) = H_{k+1}+A_{k+1}^\ast {X} (I+ G_{k+1}{X})^{-1} A_{k+1},
 \end{align*}
  $A_k$, $G_k$, and $H_k$, for $k = 1,2,\ldots$, are matrices given by the following iteration
\begin{eqnarray*}
\bb A_{k+1} \\ G_{k+1}\\ H_{k+1}\eb=F(\bb A_{k} \\ G_{k}\\ H_{k}\eb):=
\bb A_1\Delta_{G_{k},{H_1}}A_{k}\\G_1+A_1\Delta_{G_{k},{H_1}}G_{k}A_1^\ast\\H_{k}+A_{k}^\ast  H_1\Delta_{G_{k},{H_1}}A_{k}\eb,
\end{eqnarray*}
where $G_1=G\geq 0$, $H_1=H\geq 0$, $A_1=A$, and $\Delta_{G_{k},{H_1}} = I+G_{k}H_1$. An induction argument proves that $G_k\geq 0$ and $H_k\geq 0$ and thus $I+G_{k}H_1$ is
invertible. Therefore, the sequence of matrices $\mathbb{X}_k=\bb A_{k}^\top & G_{k}^\top & H_{k}^\top \eb^\top$ generated by
\begin{align}\label{aa5}
  \mathbb{X}_{k+1}=F(\mathbb{X}_{k},\mathbb{X}_{1}),
\end{align}
is well defined with $\mathbb{X}_1=\bb A_{1}^\top & G_{1}^\top & H_{1}^\top \eb^\top$ for $k\geq 1$.

Let $\mathbb{K}_n:=\mathbb{C}_n \times\mathbb{N}_n \times \mathbb{N}_n$.
We consider the mapping $F:\mathbb{K}_n \times \mathbb{K}_n \rightarrow \mathbb{K}_n$ as an action defined by
  \begin{align}\label{type1b}
  F(Y,Z)=\bb Z_{1}\Delta_{Y_2,Z_3} Y_{1} \\ Z_{2}+Z_{1}\Delta_{Y_2,Z_3} Y_{2}Z_{1}^\ast \\
  Y_{3}+Y_{1}^\ast Z_{3}\Delta_{Y_2,Z_3} Y_{1}\eb,
\end{align}
where $Y=\bb Y_{1}^\top & Y_{2}^\top & Y_{3}^\top \eb^\top$, $Z=\bb Z_{1}^\top & Z_{2}^\top & Z_{3}^\top \eb^\top\in\mathbb{K}_n$ and $\Delta_{Y_2,Z_3}=(I+Y_{2}Z_{3})^{-1}$. Note that $\Delta_{Y_2,Z_3}$ is well defined since all eigenvalues of $I+Y_{2}Z_{3}$ are positive. It has been proved in ~\cite[Example 4.4]{chiang20} that the binary operator $F$ satisfies the associative rule~\eqref{ar} holds, that is, the iteration $\mathbb{X}_{k+1}=F(\mathbb{X}_k, \mathbb{X}_1)$ has the semigroup property. As a consequence, ${\mathbb{X}}_{k+1}=F(\mathbb{X}_{1},\mathbb{X}_k)$ for any positive integers $k$ and we deduce that
\[
H_{k+1} = H_{1}+A_{1}^\ast H_k\Delta_{G_{1},{H_k}}A_{1}.
\]
from which we see that the sequence $\{ H_k \}$ {coincides} with the sequence $\{ X_k \}$ {generated by the fixed point iteration $X_{k+1}=R_{\rm d}(X_k)$} since $X_1=H_1=H$. Applying the discrete flow property~\eqref{DF} to the iteration \eqref{aa5} we obtain the following accelerated fixed point iteration (AFPI).

\begin{Algorithm} \label{aa3}
{\emph{(An accelerated of fixed-point iteration (AFPI) to solve DARE.~\eqref{eq:NMEP})}}
\begin{enumerate}

\item {Given a positive integer $r>1$, let $\widehat{\mathbb{X}}_1={\mathbb{X}}_1=\bb A_{1}^\top & G_{1}^\top & H_{1}^\top \eb^\top$;}
\item {For} $k= 1,\ldots,$ iterate
 \begin{align*}
\widehat{\mathbb{X}}_{k+1}& =F(\widehat{\mathbb{X}}_k,\mathbb{X}_{k}^{(r-1)}),
\end{align*}
    until convergence, where $F$ is defined in \eqref{type1b} and $\mathbb{X}_{k}^{(r-1)}$ is defined in step 3.
\item
     {For} $\ell=1,\ldots,r-2$, iterate
   \begin{align*}
\mathbb{X}_{k}^{(\ell+1)}& =F(\widehat{\mathbb{X}}_k,\mathbb{X}_{k}^{(\ell)}),
\end{align*}
with $\mathbb{X}_{k}^{(1)}=\widehat{\mathbb{X}}_k=\bb \widehat{A}_{k}^\top & \widehat{G}_{k}^\top & \widehat{H}_{k}^\top \eb^\top$.
 \end{enumerate}
\end{Algorithm}
Concerning Algorithm~\ref{aa3}, worth mentioning is that it can be reduced to SDA iteration when $r=2$~\cite{WWL2006,Huang2018}. The following result provides the same sufficient conditions in Theorem~\ref{solution} that guarantees the R-superlinear convergence of the sequence $\{\widehat{H}_k\}$.
\begin{Theorem}\label{lem:conv}
Assume that $\mathbb{R}_{\geq}\neq \phi$ and $H\neq 0$. Then, the sequence $\{\widehat{H}_k\}$ generated by Algorithm~\ref{aa3} converges R-superlinearly to the minimal positive semidefinite solution $X_\star$ of \eqref{eq:NMEP}. Moreover,
the convergence rate can be shown as the following.
\[
\underset{k\rightarrow \infty}{\lim {\rm sup}} \sqrt[r^k]{\|\widehat{H}_k-X_\star\|}\leq \max\{|\lambda|^2; \lambda\in\sigma(T_{{X}_{\star}})\cap {\mathbb{D}}\}.
\]
\end{Theorem}
\begin{proof}
From the above discussion, the sequence $\{\mathbb{X}_{k}\}$ has the semigroup property. Thus, $\{\mathbb{X}_{k}\}$ satisfies the discrete flow property~\eqref{DF}, which together with the construction of $\widehat{H}_{k}$ lead to
{
\[
\widehat{H}_{k}={H}_{r^{k}}={X}_{r^{k}}
\]
}
 for all integers $k\geq 1$~\cite{chiang18}[Remark 4.1]. The proof is straightforward from Theorem~\ref{solution}.
\end{proof}
\subsection{Numerical examples}
In this subsection, all computations were performed in MATLAB/version 2016a on a PC with an Intel Core i5-8279U GHZ processor and 8 GB main memory, using IEEE double-precision floating-point arithmetic ( eps $=2^{-52}\approx 2.22\times 10^{-16}$).
\begin{example} Let $A=A_{1}\bigoplus A_{2}$, $G=G_{1}\bigoplus G_{2}$ and $H=H_{1}\bigoplus H_{2}$, where $(A_i,G_i,H_i)$ is a set of matrix
coefficients corresponding to a DARE~\eqref{eq:NMEP} for $i=1$ or $2$. Consider the first set of matrix coefficients $(A_1,G_1,H_1)$ depending on
four parameters and is defined as
 $A_{1}=\left(
         \begin{array}{ccc}
           \epsilon & 1 &0 \\
           0 & 0 & 0\\
           0 & 0 & 0
         \end{array}
       \right),
       G_{1} =\left(
         \begin{array}{ccc}
           g &0 & 0 \\
           0 & 0 &0 \\
           0 & 0 & 0
         \end{array}
       \right),
$ $H_{1}=\left(
         \begin{array}{ccc}
           0 & 0 & 0\\
           0 & a & b\\
           0 & \bar{b} & c
           \end{array}
       \right)$,
 where $a,c,g,\epsilon>0$ and $b\in\mathbb{C}$ satisfies $ac\geq |b|^2$.
    The second set of matrix coefficients $(A_2,G_2,H_2)$ is constructed by applying some MATLAB
functions according to the following steps
  $A_{2}=\verb!crand!(2,2)$,$G_{2}=U^\ast\verb!diag!([2,1])U$
 and $H_{2}=V^\ast\verb!diag(rand(2,1))!V$, where $U=\verb!orth(rand(2,2))!$ and $V=\verb!orth(rand(2,2))!$.

  \end{example}
It can be shown that $\rho(T_{H_1})=|\epsilon|$, $G_1H_1=0$ and $X=H_1$ solves DARE~\eqref{eq:NMEP} with coefficients $(A_1,G_1,H_1)$ and thus has a minimal positive semidefinite solution $X=X_1$ by Theorem~\ref{solution}, while applying Corollary~\ref{thm:pd1} the minimal positive semidefinite solution $X_2$ of DARE~\eqref{eq:NMEP} with coefficients $(A_2,G_2,H_2)$ exists since $G_2>$ and $\rho(T_{X_2})<1$. We conclude that $X_m=X_1\bigoplus X_2$ is the minimal positive semidefinite solution of Eq.~\eqref{eq:NMEP} with coefficients $(A,B,C)$.
\begin{figure}
\begin{center}
       \includegraphics[width=80mm]{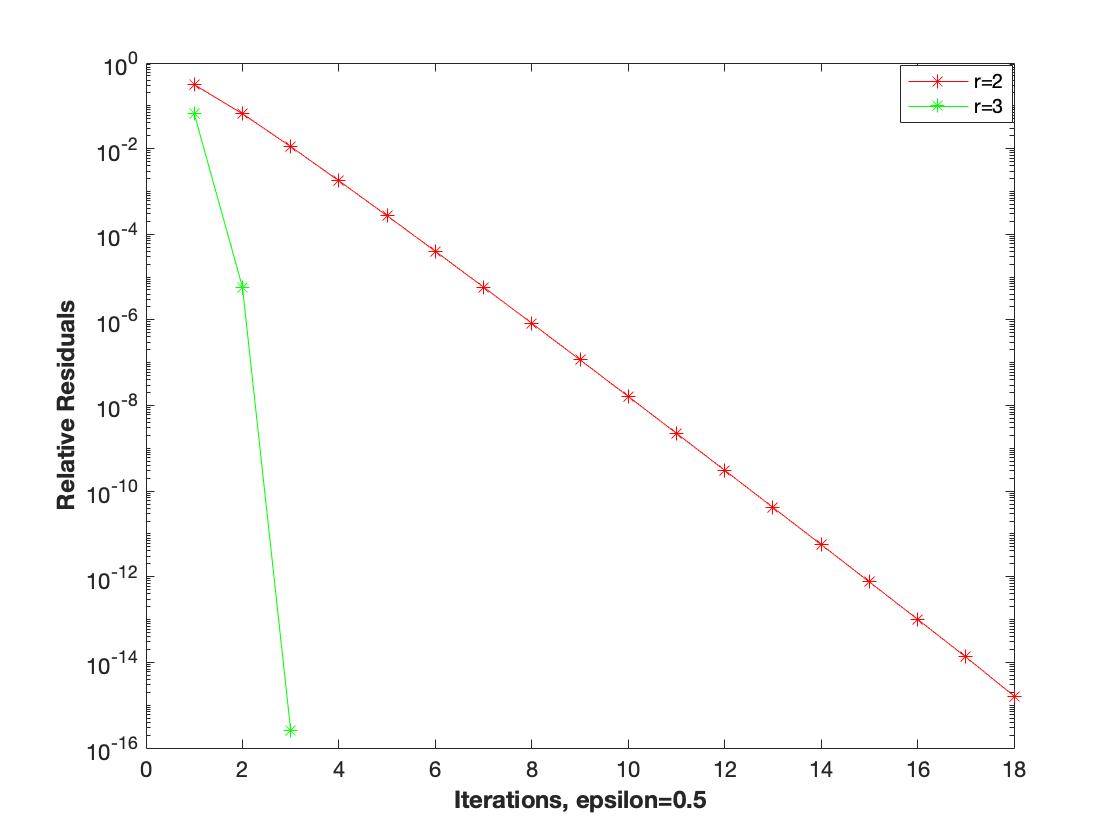}
       \includegraphics[width=80mm]{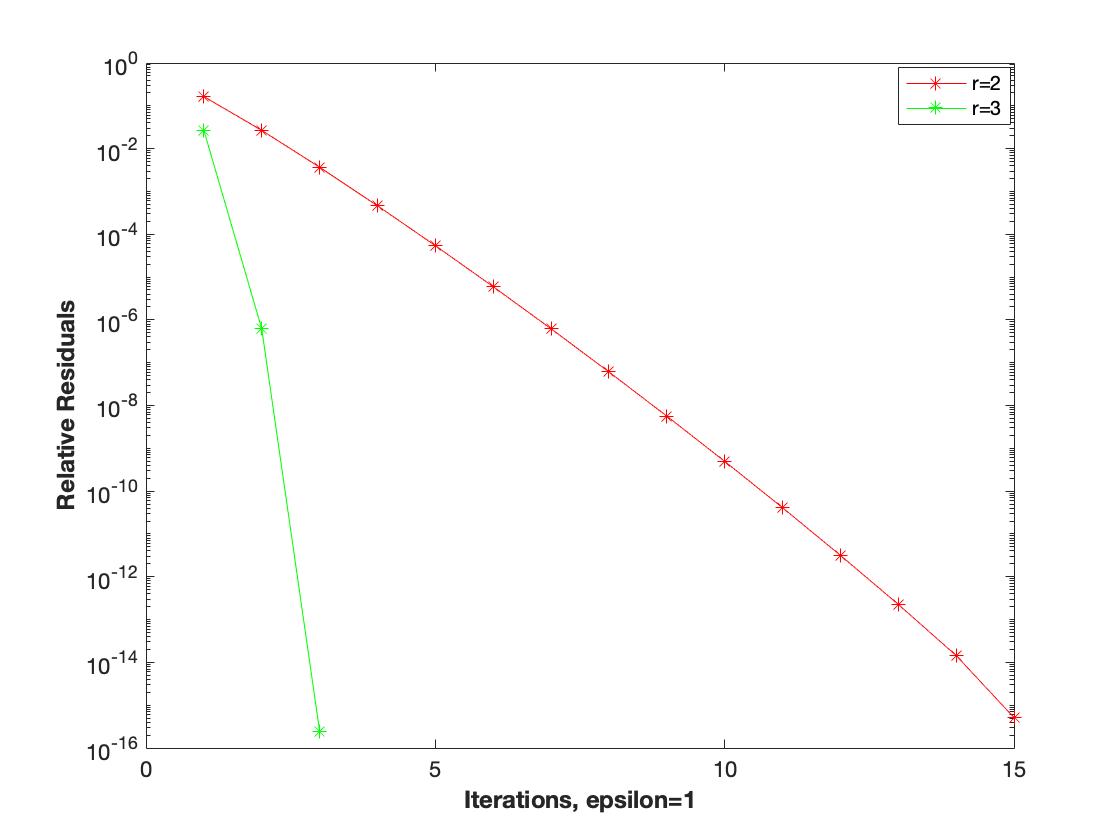}
       \includegraphics[width=80mm]{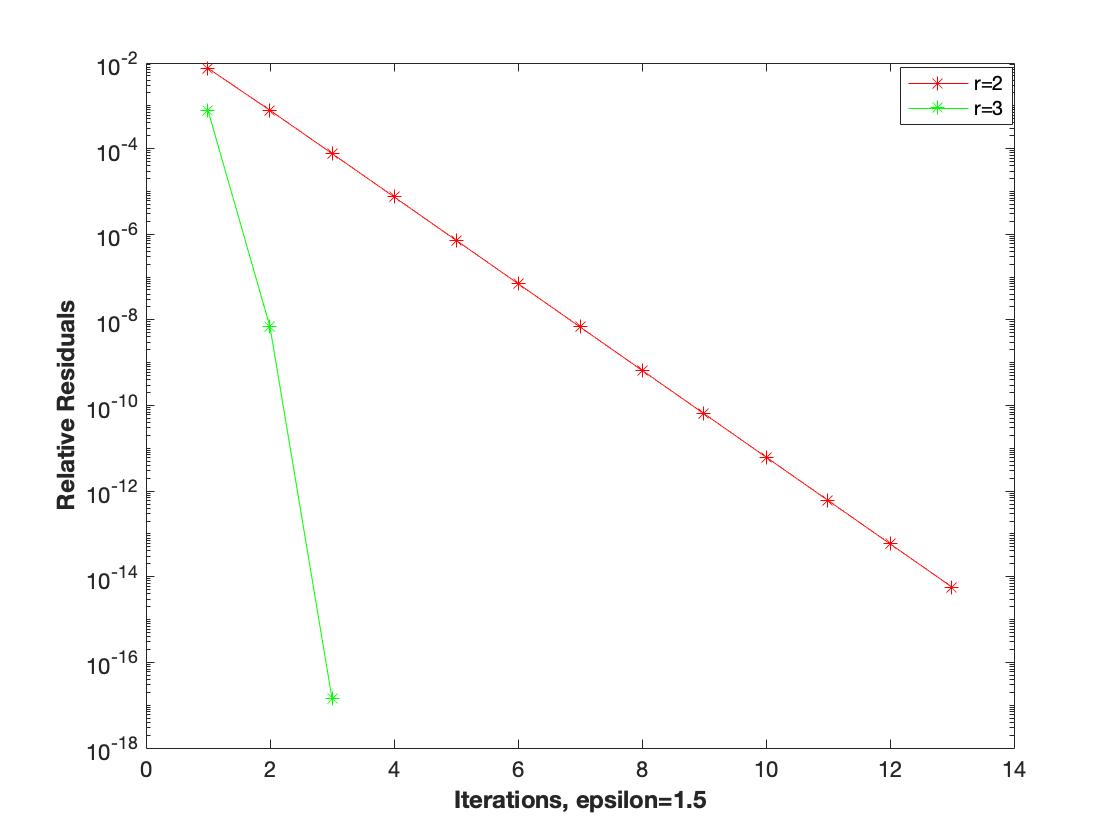}
        \end{center}
\caption{\label{epsilon}Convergence behaviour of AFPI iteration with $\epsilon=0.5, 1, 1.5$, respectively.}
\end{figure}
Now, we set $g=2$, $a=4, c=1$ and $b=\verb!crand!$. It can be easily proved that $\rho(T_{X_m})=\max(\rho(T_{X_1}),\rho(T_{X_2}))$, then $\rho(T_{X_m})=\rho(T_{X_1})=\epsilon\geq 1$ if $|\epsilon|\geq 1$ and $\rho(T_{X_m})<1$ if $|\epsilon|<1$. By choosing $\epsilon=0.5, 1, 1.5$, respectively, Figures~\ref{epsilon} show the efficiency of the accelerated algorithm with $r=2, 3$, respectively.

\begin{example}
This example is inspired by \cite{HUANG20091452}[Example 5.2]. Let 	
$$A=\left(\begin{array}{ccc}
	-1 &-1/2\\
	0 & -1\\
\end{array}\right)\bigoplus \left(\begin{array}{ccc}
	-1 & -1/2 & -1/8\\
	0 & -1 & -1/2\\
	0 & 0 & -1\\
\end{array}\right)\bigoplus \left(-\frac{2}{3}I_2\right),
$$
$$G=\left(\begin{array}{ccc}
\frac{1}{32}& \frac{1}{8}\\
\frac{1}{8}& \frac{1}{2}\\	
\end{array}\right)\bigoplus \left(\begin{array}{cccc}
\frac{1}{512}& \frac{1}{128} & \frac{1}{32}\\
\frac{1}{128} & \frac{1}{32}& \frac{1}{8}\\
\frac{1}{32} &\frac{1}{8}&\frac{1}{2}\\	
\end{array}\right) \bigoplus \left(\frac{2}{9}I_2\right),
$$ and $H=-3.5I_7$.
\end{example}
 If $r=1$, that is, the original fixed point iteration, does not converge to the solution within 1000 iterations, while for $r\geq 3$, the accelerated iteration converges to a negative solution linearly. Figure \ref{fig:neg} reports the convergence behaviour of the accelerated iteration with $r=3,4,5,6,7$.
{ In this example, $H$ is a negative definite matrix. In other words, the assumption for the positivity of $H$ is not satisfied. However, the similar convergent result of AFPI with $r\geq 3$ under weaker conditions appears in Figure~\ref{fig:neg}. By the way, the authors illustrate the superior performance of AFPI with $r=2$ as compared to NM and matrix disk function method (MDFM) in the numerical experiments (\cite{HUANG20091452}[Example 5.2]), which show that AFPI with $r=2$ perform feasibility and reliably. We believe AFPI may still converge even if $H$ is indefinite. How to apply the accelerated techniques in the work under other suitable conditions leads to the work in future.}

\begin{figure}
\begin{center}
       \includegraphics[width=90mm]{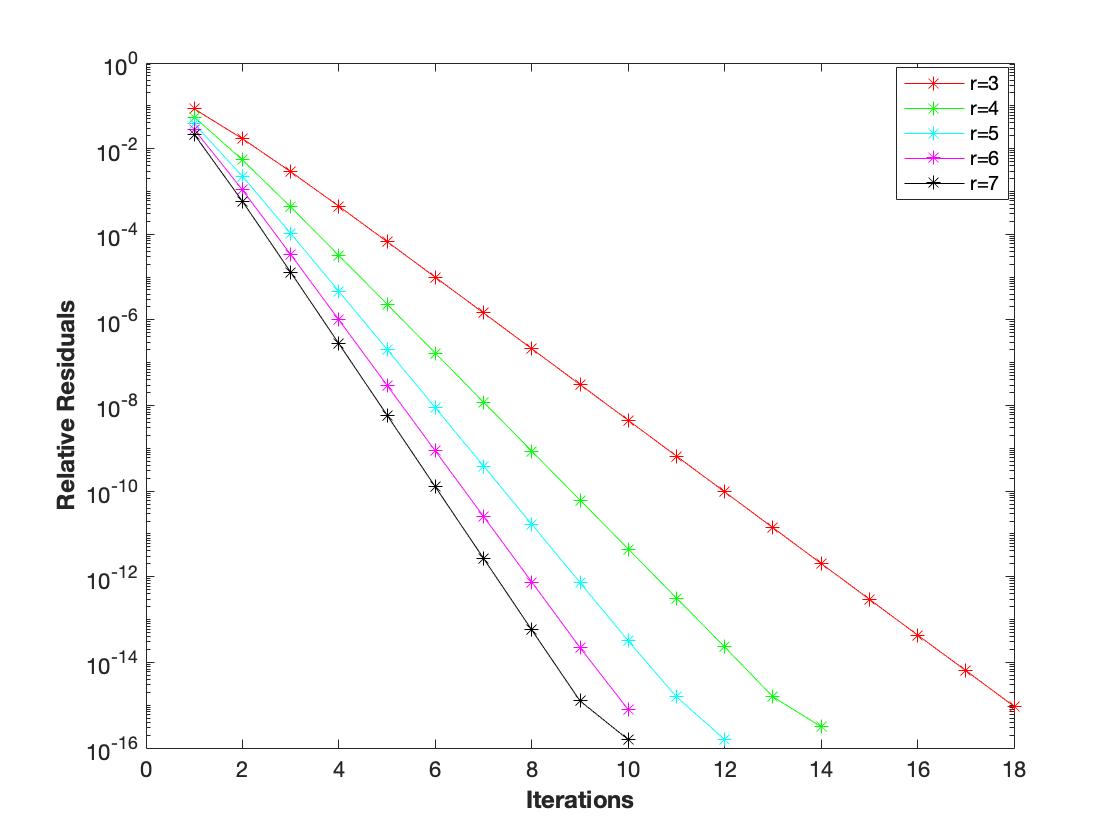}
\end{center}
\caption{\label{fig:neg}Convergence behaviour of APFI with $r=3,4,5,6,7,8$, respectively.}
\end{figure}

\section{Concluding remarks}
This paper concerns comprehensive convergence analysis of the most recent and advanced algorithms including SDA and its variants AFPI for solving DARE \eqref{eq:NMEP}. Our contribution fills in an existing gap in the minimal positive semidefinite solution $X_\star$ of the DARE~\eqref{eq:NMEP}, concerning the magnitude $\rho(T_{X_\star})\geq 1$. More precisely, we have proved the convergence for the AFPI, when the eigenvalues of $T_{X_\star}$ are inside, on or outside the open unit circle. The theoretical result is confirmed by a randomness numerical example. Consequently, our results are more general than those in the past works, which considered only eigenvalues lies in the closed unit disk. The techniques of Proposition 2.1 and Theorem~\ref{solution} can be adopted in the convergence analysis of AFPI. We believe the results we obtain are novel {on} this topic and could provide considerable insights into the study of other nonlinear matrix equations.

\section*{Acknowledgment}
The author wish to thank Dr.Jie Meng and four anonymous referees for many interesting and valuable suggestions on the manuscript.
This research work is partially supported by the Ministry of Science and Technology and the National Center for Theoretical Sciences in Taiwan.
The author would like to thank the support from the Ministry of Science and Technology of Taiwan under the grant MOST 108-2115-M-150-002.
\section*{Appendix : An alternating iteration}
\begin{proof}
\begin{enumerate}
  \item First, for a positive definite matrix $Q$ it is fairly easy to see that $E_k:=\sum\limits_{j=0}^{k-1} (A^k)^\ast Q A^k$ is a Cauchy sequence if and only if $\rho(A)<1$.

   Assume that $S_A(X_0)>0$ for some $X_0>0$.  Then, there exists a positive number $\epsilon$ such that $X_0> \epsilon I_n+A^\ast X_0 A\geq\epsilon \sum\limits_{j=0}^k (A^j)^\ast A^j$ for any positive integer $k$. It immediately implies that $\rho(A)<1$. Conversely, $X=X_0:=\sum\limits_{j=0}^\infty (A^k)^\ast Q A^k>0$ solves the equation $S_A(X)=Q$ if $\rho(A)<1$, where $Q$ is any positive definite matrix.
  \item  Assume that there exists a $X\in \mathbb{P}_n$ such that $Q:= S_A(X)\geq 0$. It can be shown that there exists a nonsingular matrix $S\in\mathbb{C}^{n\times n}$ such that $S^\ast X S=I_n$ and $S^\ast Q S=Q_r\oplus 0_{n-r}$ with a $r\times r$ positive diagonal matrix $Q_r$~\cite{Bernstein2009}[Theorem 8.3.1]. We transform \eqref{SME} into the following equation
 \[
 0\leq \widehat{A}^\ast\widehat{A}=(I_r-Q_r)\oplus I_{n-r} \leq I_n,
 \]
 where $\widehat{A}=S^{-1}A S$. It is immediately that 
 $\widehat{A}$ is semicontractive~\cite{Bernstein2009}[Definition 3.1.2] and thus $A$ is discrete-time Lyapunov stable~\cite{Bernstein2009}[Fact 11.21.4]. Namely, $\rho ({A}) \leq 1$, and if $\rho ({A}) = 1$, then all unimodular eigenvalues of ${A}$ is semisimple.

Conversely, if there exists nonsingular matrix $P$ such that  $PA P^{-1}=J_s\oplus J_1$ where $J_s$ is a Jordan canonical form of $k\times k$ matrix satisfying $\rho(J_s)<1$ and $J_1$ is a $(n-k)\times (n-k)$ diagonal matrix satisfying $\rho(J_1)=1$. Let $\widehat{X}:= P^\ast X P$ and $\widehat{Q}:= P^\ast Q P$. We partition two matrices
$\widehat{X}$ and $\widehat{Q}$ as $2\times 2$ block matrices $\bb \widehat{X}_{i,j} \eb$ and $\bb \widehat{Q}_{i,j} \eb$, respectively, where $\widehat{X}_{1,1},\widehat{Q}_{1,1}\in\mathbb{C}^{k\times k}$ and $\widehat{X}_{2,2},\widehat{Q}_{2,2}\in\mathbb{C}^{(n-k)\times (n-k)}$. We have
\begin{subequations}\label{AP}
 \begin{align}
  \widehat{X}_{1,1} &=\widehat{Q}_{1,1}+J_s^\ast\widehat{X}_{1,1} J_s,\label{AP1}\\
  \widehat{X}_{2,1} &=\widehat{Q}_{2,1}+J_s^\ast\widehat{X}_{2,1} J_1,\label{AP2} \\
  \widehat{X}_{2,2} &=\widehat{Q}_{2,2}+J_1^\ast\widehat{X}_{2,2} J_1.\label{AP3}
\end{align}
\end{subequations}
It is obtained $\widehat{Q}=I_k\oplus 0_{n-k}$  by choosing $Q=P^{-H}(I_k\oplus 0_{n-k})P^{-1}$. It is easily to check that there exist two unique solution $\widehat{X}_{11}\in\mathbb{P}_n$ and $\widehat{X}_{21}=0_{k\times(n-k)}\in\mathbb{C}^{k\times(n-k)}$ to the corresponding Stein matrix equations \eqref{AP1} and \eqref{AP2}. Let
$\widehat{X}_{22}$ be a diagonal matrix whose diagonal elements are all nonnegative. Then, the matrix $X=P^{-H}(\widehat{X}_{11}\oplus\widehat{X}_{22})P^{-1}\in\mathbb{N}_n$ satisfies $S_A(X)=Q\geq 0$. Moreover, a positive definite matrix $X$ can be chosen by setting $\widehat{X}_{22}>0$.
\end{enumerate}
\end{proof}
%

%
%

\def\cprime{$'$}

\end{document}